\documentclass{amsproc}
\usepackage{graphicx}
\usepackage{latexsym}
\usepackage[all]{xy}
\usepackage{color}
\usepackage{cancel}
\usepackage{amssymb, mathrsfs, amsfonts, amsmath}
\newtheorem{theorem}{\bf{Theorem}}[section]
\newtheorem{lemma}[theorem]{\bf{Lemma}}
\newtheorem{corollary}[theorem]{\bf{Corollary}}

\theoremstyle{definition}

\numberwithin{equation}{section}

\newcommand{\ds}{\displaystyle}

\begin{document}
\title[  Rigidity Theorems for Anosov Geodesic Flows]{Some Rigidity Theorems for Anosov Geodesic Flows  in Manifolds of Finite Volume }
  

\author{\'Italo Melo}
\address{Departamento de
	Matem\'{a}tica-UFPI, Ininga, 
	64049-550 Teresina, Brazil.}
\thanks{$^{2}$ Partially supported by CAPES-BR}
\email{italodowell@ufpi.edu.br}
\ \\
\indent\author[Sergio Roma\~na]{Sergio Roma\~na}
\indent \address{Instituto de Matemática, Universidade Federal do Rio de Janeiro, CEP 21941-909, Rio de Janeiro, Brazil and  Department of Mathematics, Southern University of Science and Technology, Shenzhen, China.}
\thanks{$^{1}$ CNPq and Bolsa Jovem Cientista do Nosso Estado No. E-26/201.432/2022}
\email{sergiori@im.ufrj.br}
\begin{abstract}
In this paper, we prove that if the geodesic flow of a complete manifold  without conjugate points with sectional curvatures bounded below 
by $-c^2$ is of Anosov type, then the constant of contraction of the flow is $\geq e^{-c}$. Moreover, if $M$ has a finite volume,
the equality holds if and only if the sectional curvature is constant. We also { apply} this result {to get} a certain rigidity for bi-Lipschitz, and consequently, for $C^1$-conjugacy between two geodesic flows. 
\end{abstract}
\maketitle
\section{Introduction and Main Results}
Geodesic flows appear naturally when we have a Riemannian metric on a complete manifold. Its properties are closely linked with the geometry of 
the Riemannian metric.\\
\indent From Hadamard's work on cutting sequences and Hopf's work, which proved the ergodicity of geodesic flows on surfaces of negative curvature \cite{hopf:hopf:0}, we know that geometry influences the chaotic properties of geodesic flow. This property was further explored by Anosov in his seminal work \cite {anosov:anosov}, where he proved that geodesic flows of negative curvature manifolds are uniformly hyperbolic systems. Currently, we simply call them “\emph{Anosov}” systems. The dynamic properties of the geodetic flow of negative curvature manifolds have been studied for many decades. To date, the study of the ergodic and dynamic properties of geodesic flow is a very active topic. In \cite{anosov:anosov}, Anosov began the task of knowing which types of geometries exhibit the Anosov property.


\indent In \cite{eb:eb}, Eberlein determined geometric conditions that are equivalent to the Anosov pro\-per\-ty of the geodesic flow when $M$ is
compact or  compactly homogeneous \emph{i.e.}, the isometry group of its universal cover acts co-compactly. More specifically, Eberlein's
conditions give information related to the divergence of Jacobi fields or about the geodesics passing  through points of negative sectional 
curvature. In \cite{bol:bol}, Bolton observed that the condition of compactly homogeneous by Eberlein at \cite{eb:eb} is unnecessary. 

In the opposite direction, we can ask the following question: What geometric properties impose the geodesic flow to be Anosov?\\
\indent  A satisfactory answer to this question was given by Klingenberg, who proved  that a compact Riemannian manifold with the geodesic
flow of Anosov type has no conjugate points (cf. \cite{kli:kli}). This result was generalized by Ma\~n\'e in \cite{mane:mane} for the case of finite volume.   Recently, the infinite volume case was proved by the authors in \cite{RomanaMelo3}.

 \indent The results of Klingenberg at \cite{kli:kli}, Ma\~n\'e at  \cite{mane:mane}, and Melo-Roma\~na at \cite{RomanaMelo3} show that some geometric properties of manifolds with Anosov geodesic flow are consequences of dynamical properties. In the same spirit, in this paper, we will show some results which give an important relation between the geometry 
and dynamical behavior of the geodesic flow. Moreover, we will obtain that a  dynamical rigidity implies geometric rigidity (cf. Theorem \ref{curvature}), which will be used to show rigidity in the level of conjugacy  between two geodesic flows (cf. Theorem \ref{R-conjugation} and Theorem \ref{R-conjugation-WCP}).

 To state our results, we start with the precise definition of Anosov geodesic flow for general case, even non-compact manifold:
 Let $M$ be a complete Riemannian manifold and $SM$ the unitary tangent bundle, endowed with the \emph{Sakaki} metric (see Section \ref{GeoFlow}). 	Let $\phi^{{t}}_{_{M}}:SM \rightarrow SM$ be the geodesic flow, and 
suppose that $\phi^t_{_{M}}$ is \emph{Anosov}, this means that  the tangent bundle of $SM$, $T(SM)$,  has a splitting $T(SM) = E^s \oplus \langle G \rangle \oplus E^u $ such that 
\begin{eqnarray*}
	d(\phi^t_{_{M}})_{\theta} (E^s(\theta)) &=& E^s(\phi_{_{M}}^t(\theta)),\\
	d(\phi^t_{_{M}})_{\theta} (E^u(\theta)) &=& E^u(\phi_{_{M}}^t(\theta)),\\
	||d(\phi^t_{_{M}})_{\theta}\big{|}_{E^s}|| &\leq& C \lambda^{t},\\
	||d(\phi^{-t}_{_{M}})_{\theta}\big{|}_{E^u}|| &\leq& C \lambda^{t},\\
\end{eqnarray*}
for all $t\geq 0$ with $ C > 0$ and $0 < \lambda <1$, where the constant $\lambda$ is called a \emph{constant of contraction} of the flow and $G$ is the geodesic vector field. In our first result, we give a lower bound for the possible values of $\lambda$ depending on the lower bound of the curvature. Furthermore, we prove that if $\lambda$ reaches its minimum value, then we have a rigidity of the geometry. More specifically
\begin{theorem}\label{curvature}
	Let $M$ be a complete Riemannian manifold with sectional curvature bounded below by $-c^2$ and Anosov geodesic flow  $\phi^{t}_{_M}$.
	\begin{itemize} 
	\item[(a)] Any constant of contraction $\lambda$ of $\phi^{t}_{_M}$ satisfies $\lambda \geq e^{-c}$.
	\item[(b)] If $M$ has finite volume, then $\lambda=e^{-c}$  if and only if the sectional curvature of $M$ is constant, equal to $-c^2$. 
\end{itemize}	
\end{theorem}
The first part of Theorem \ref{curvature} implies that for  Anosov geodesic flows the dynamics is controlled, to some extent, by the curvature. 
The second part is more general: it says that rigidity in the dynamical sense provides rigidity in the geometrical sense. 
We emphasize that our result does not require the compactness of the manifold.

It is worth noting that the first part of Theorem \ref{curvature} follows from the proof of the corollary of Theorem A by Ma\~ne  in  
\cite{mane:mane}, this result can also be proved using the Rauch comparison theorem, for more details see \cite{handbook:handbook}. However, Ma\~ne's proof and the approach using the Rauch comparison theorem do not provide any information concerning rigidity.
Here we will give a different proof of this theorem in which, if the equality holds, the additional information obtained in the proof provides a new understanding of the rigidity problem.\\
\indent Another important thing is to observe that in the compact case, the rigidity of Theorem \ref{curvature}  can be proved using  Pesin's formula for the entropy and Lyapunov exponents and some techniques of \cite{mane:freire}, but it is  worth noting that, in general,  Pesin's formula  is not valid  in non-compact manifolds (cf. \cite{fel:riq}), so as some theorems in \cite{mane:freire}. As we are interested also in the non-compact case, our proof does not use Pesin's formula.

\indent In the proof of the first part of Theorem \ref{curvature}, the hypothesis of finite volume is required only to ensure the non-existence of conjugate points (see \cite{mane:mane}). Therefore, when the volume is infinite and without conjugate points, the first part of this theorem is also valid, see Lemma \ref{inequality curvature}. It is worth mentioning that the condition of no conjugate points is redundant due to the authors' recent preprint (see \cite{RomanaMelo3}). 


\indent To state the other results concerning the regularity of conjugacy between two geodesic flows, we need the following definition.

\indent We say that the two flows $\varphi^t \colon N\to N$ and $\eta^t\colon S \to S$ are \emph{equivalents} if there is a continuous map $f\colon N\to S$ such that $f\circ \varphi^t=\eta^t\circ f$. The map $f$ is called an \emph{equivalence}. Equivalences between two systems are very important in  dynamical systems because  relevant information about one system is transferred to another by the equivalences.

When an equivalence is a homeomorphism it is called a \emph{conjugacy}. 
An equivalence (conjugacy) $f$ is called $1$-equivalence ($1$-conjugacy) if $f$ is bi-lipschitz {\footnote{We can think in $\alpha$- equivalence ($\alpha$-conjugacy), when there are two constants $C^1$ and $C^2$ such that $C_1\cdot d(x,y)^{\alpha}\leq D(f(x),f(y))\leq C_2\cdot d(x,y)^{\alpha}$, but such definition do not make any sense for $C^1$ flows.}, \emph{i.e.}, there are two constants $C_1$ and $C_2$ such that 
\begin{equation}\label{e1111}
C_1\cdot d(x,y)\leq D(f(x),f(y))\leq C_2\cdot d(x,y),
\end{equation}
where $d$ and $D$ are the distances in $N$ and $S$, respectively.\\
In the next theorem, we obtain a rigidity result between two $1$-equivalent flows, where one of them is an Anosov geodesic flow.


\begin{theorem}\label{R-conjugation}
Let $M$, $N$ be two complete Riemannian manifolds such that sectional curvatures satisfy  $\inf K_M \geq \sup K_{N} \geq K_N\geq -b^2$, for some $b>0$. Assume that $M$ has finite volume and  $\phi^t_{_{M}}$ is Anosov. If  $\phi^t_{_{M}}$ and $\phi^t_{_{N}}$ are $1$-equivalent, then $K_M\equiv\sup K_N<0$.
\end{theorem}
It is important to note that in Theorem \ref{R-conjugation} the compactness of $M$ and $N$ is not required. But, when $M$ is a compact manifold, then Theorem \ref{R-conjugation} has an important consequence.  We say that a map $f\colon N\to S$ is an \emph{immersion} if for all $x\in N$ it holds that $$Df_{x}(v)=0 \, \,\, \text{if and only if} \ \ \, v=0.$$
\begin{corollary}\label{C0}
Let $M$ be a compact manifold and $N$ a complete manifold. Assume that $\phi^t_{_{M}}$ is Anosov and $\inf K_{M}\geq \sup K_N\geq K_N \geq -b^2$, for some $b>0$.  Then, if $\phi^t_{_{M}}$ and $\phi^t_{_{N}}$ are $C^1$-equivalent by a $C^1$ injective immersion,  then  $K_M\equiv \sup K_N<0$.
\end{corollary}
The above result imposes a certain rigidity for regular {equivalences} of two Anosov geodesic flows. Observe that the only condition about the dimension of $M$ and $N$ is $\text{dim}M\leq \text{dim}N$, since $f$ is an immersion. Therefore, 
if $\text{dim}M=\text{dim}N$, then $f$ is a $C^1$-diffeomorphism and  $N$ should be compact. However, if  $N$ is a compact manifold and $\text{dim}\,M=\text{dim}\,N$ we have a more general result by mixing the proofs of Theorem \ref{curvature} and Theorem \ref{R-conjugation}. 

\begin{theorem}\label{R-conjugation-WCP}
Let $M$, $N$ be two compact  Riemannian manifolds with the same dimension such that $\inf K_M \geq \sup K_{N}$, and $M$ has no conjugate points. If $\phi^t_{_{M}}$ and $\phi^t_{_{N}}$ are $1$-conjugate,  then $K_M\equiv K_{N} \equiv {\sup K_N}$.
\end{theorem}

Since compact manifolds with Anosov geodesic flow have no conjugate points (cf. \cite{kli:kli}), an immediate consequence is
\begin{corollary}\label{C01}
Let $M$, $N$ be two compact  Riemannian manifolds with the same dimension such that $\inf K_M \geq \sup K_{N}$ and  $\phi^t_{_M}$ is Anosov. If $\phi^t_{_{M}}$ and $\phi^t_{_{N}}$ are $1$-conjugate,  then $K_M\equiv K_{N} \equiv {\sup K_N}$.
\end{corollary}

\noindent The previous results are related to the following conjecture (cf. Conjecture 5.2.1  at  \cite{Burns}):\\
\ \\
\indent \textbf{Conjecture:} \emph{Compact Riemannian manifolds with negative curvature must be isometric if they have $C^0$-conjugate geodesic flows.}\\
\ \\
Note that, in particular, the hypothesis of Theorem \ref{R-conjugation-WCP} assures that there is no $1$-conjugacy unless the curvature of both manifolds is constant, which proves the conjecture for $1$-conjugacy (consequently for  $C^1$-conjugacy) and some relation of the curvatures.\\
\indent The latter results are in the direction of recent research related to the smoothness of conjugacy for geodesic flows in compact manifolds. 
Matters related to rigidity in the conjugacy of geodesic flows have been and are being much studied  recently. Let us cite some related articles. 
The first is an article by three authors, Gallot, Besson, and Courtois (see \cite{GBC1996}), which among other things, proves that if the geodesic flow of a compact manifold locally symmetric and negative curvature is $C^1$-conjugate to the geodesic flow of another  manifold, then the metrics are isometric (see also \cite{Distrib_Lyap_Diff_1990} for other results about the regularity of conjugacy). In \cite{Croke} the conjecture was shown for manifolds of $\emph{rank}\geq 2$.  Also,  J. Feldman and D. Ornstein showed that the conjugacy between two geodesic flows of surfaces of negative curvature is $C^1$, which was generalized by Pollicott (see \cite{Pollicott1990}) showing that the conjugacy is $C^\infty$. Note that the manifold $M$ in Theorem \ref{R-conjugation-WCP} does not have necessarily negative curvature, therefore that result of Feldman and Ornstein  does not apply. 
Feres and Katok (see \cite{Feres_katok_1989}), proved that when the horospheric foliation is smooth and the 
curvature is negative, then the flow is $C^1$-conjugate to a
geodesic flow of a manifold of constant negative curvature (see \cite{Kanai} for the smooth case). To finish our references, there is a result of C. Connell 
(see \cite{Connell_2003}) which assumes a relation of curvatures between two compact manifolds one of them being locally symmetric, 
then he shows that Lipschitz semi-conjugacy  between their geodesic flows makes the other manifold also locally symmetric. In other words,
in the locally symmetrical context, we have rigidity for conjugacy. Here we observe that, in this work, all manifolds involved are not necessarily locally symmetric, which makes our results interesting ({see \cite{Croke2} for more results about rigidity in Riemannian manifolds}).

In Section \ref{proof of T-R-C},  we will prove the Theorems  \ref{R-conjugation}  and \ref{R-conjugation-WCP} and their corollaries. %

\noindent We conclude this introduction by showing the relations of the techniques developed in this work with the scenario of Lyapunov exponents.

In \cite{butler:butler}, Butler studied the rigidity results of Lyapunov exponents for geodesic flows on a compact negatively curved manifold. He proved that if each periodic orbit of the geodesic flow has the same Lyapunov exponent over the unstable bundle, then the manifold has constant negative sectional curvature. 

Since the geodesic flow on compact manifolds of negative curvature is of Anosov type, then by  Butler's result we can ask the following:\\
\ \\
\textbf{Conjecture 1:} \emph{Let $M$ be a complete Riemannian manifold with finite volume and Anosov geodesic flow. If the unstable Lyapunov exponents are constant in
all periodic orbits, then $M$ has constant negative sectional curvature.}\\
\ \\
Let $M$ be a manifold with curvature bounded below. For each $x\in M$, let $T_xM$ be the tangent space at $x$. So, for each plane $P\subset T_{x}M$, we denote by $K_{x}(P)$ the sectional curvature of the plane $P$. If the geodesic flow $\phi^t_{_{M}}$ is Anosov then, from Corollary \ref{NonNegCur}, we define the positive and finite number $c_{_{\inf}}$ as follows,  $$-c_{_{\inf}}^2:=\ds\inf_{\substack{x\in M \\ P\in T_x M }}K_{x}(P).$$
Therefore, using the techniques developed to prove  Theorem \ref{curvature}, we reduce Conjecture 1 to the following conjecture (see Corollary \ref{Conj2-Conj1}).\\
\ \\
\textbf{Conjecture 2:} \emph{Let $M$ be a complete Riemannian manifold with finite volume, sectional curvature bounded below, and Anosov geodesic flow. If the unstable Lyapunov exponents in all periodic orbits are constant equal to $a$, then $a=c_{_{\inf}}$.}
\ \\

It was an open problem that the Anosov condition in the case of infinite volume does not imply the condition of no conjugate points (cf. \cite{handbook:handbook}). However, the authors in \cite{RomanaMelo3}, recently  gave a positive answer to this fact. Thus, the proof of Theorem \ref{curvature} will be carried out in two steps. In the first phase, we will prove the inequality relating any constant of contraction of the flow with the exponential of the bound of the curvature, assuming no conjugate points and Anosov condition, without any assumption about the volume of the manifold, see Subsection \ref{inequality curvature} (or from \cite{RomanaMelo3} only  assuming the Anosov condition). In the second phase, we will prove that the Lyapunov exponents are constants for every point and, finally, using geometric arguments associated with geodesic flows 
in manifold without conjugate points, we will conclude that the curvature is constant. For the proofs of Theorem \ref{R-conjugation} and Theorem \ref{R-conjugation-WCP},  we will use the equivalence to transfer all information of the hyperbolicity from $\phi^t_{_{N}}$ to $\phi^t_{_{M}}$ and thus use the techniques of Theorem \ref{curvature}.




\section{Notation and Preliminaries}

Throughout this paper, $M=(M,g)$ will denote a boundaryless complete Riemannian manifold of dimension $n\geq 2$, $TM$  its tangent bundle, $SM$ 
its unit tangent bundle. The mapping $\pi\colon TM\to M$ will denote the canonical projection and $\mu$ the Liouville measure of $SM$. All settings of this section can be found in \cite{gabriel:paternain} (see also \cite{handbook:handbook}). 


\subsection{Geodesic flow }\label{GeoFlow}

Let $\theta=(p,v)$ be a point of $SM$ and let $\gamma_{\theta}(t)$ be the unique geodesic with initial conditions $\gamma_{\theta}(0)=p$ and 
$\gamma_{\theta}'(0)=v$. For a given $t\in \mathbb{R}$, let\break $\phi^t_{_{M}}:SM \to SM$ be the diffeomorphism given by 
$\phi_{_{M}}^t(\theta)=(\gamma_{\theta}(t),\gamma_{\theta}'(t))$. Recall that this family is a flow (called the \textit{geodesic flow}) in the sense that 
$\phi^{t+s}_{_{M}}=\phi^{t}_{_{M}}\circ \phi^{s}_{_{M}}$ for all $t,s\in \mathbb{R}$. \\
Let $V:=\emph{ker}\, D\pi$ be the \textit{vertical} sub-bundle of $T(TM)$ (tangent bundle of $TM$). 
Let $\alpha\colon TTM\to TM$ be the Levi-Civita connection map of $M$. Let $H:=\text{ker}\,  \alpha$ be the \emph{horizontal} sub-bundle. Recall that, $\alpha$ is defined as
follows: Let $\xi \in T_{\theta}TM$ and\break $z:(-\epsilon,\epsilon) \rightarrow TM$  be a curve adapted to $\xi$, \emph{i.e.},  $z(0) = \theta$ and 
$z'(0) = \xi$, where\break $z(t) = (\alpha(t),Z(t))$, then
$$K_{\theta}(\xi)=\nabla_{\frac{\partial}{\partial\,t}}Z(t)|_{t=0}.$$


{For each $\theta$, the maps $d_{\theta}\pi|_{H(\theta)}: H(\theta) \rightarrow T_pM$ and $K_{\theta}|_{V(\theta)}:V(\theta) \rightarrow T_pM$ are linear
isomorphisms. Furthermore, $T_{\theta}TM = H(\theta) \oplus V(\theta)$ and the map  $j_{\theta}:T_{\theta}TM \rightarrow T_pM \times T_pM$ given by 
	$$ j_{\theta}(\xi) = (D_{\theta}\pi(\xi),K_{\theta}(\xi))        $$
	is a linear isomorphism. }

{	Using the decomposition $T_{\theta}TM = H(\theta) \oplus V(\theta)$, we can identify a vector in $\xi \in T_{\theta} TM$ with the pair of vectors in
$T_{p}M$, $(D_{\theta}\pi(\xi),K_{\theta}(\xi))$  and define naturally a Riemannian metric on $TM$
	that makes $H(\theta)$ and $V(\theta)$ orthogonal. This metric is called the Sasaki metric and is given by
	$$ g_{\theta}^S(\xi,\eta) = \langle D_{\theta}\pi(\xi), D_{\theta}\pi(\eta)  \rangle  +  \langle K_{\theta}(\xi),    K_{\theta}(\eta)   \rangle. $$}


From now on, we consider the Sasaki metric restricted to the unit tangent bundle $SM$.  It is easy to prove that the geodesic flow preserves
the volume measure generated by this Riemannian metric in $SM$. Furthermore, this volume measure in $SM$ coincides with the Liouville measure 
$m$ up to a constant. When $M$ has finite volume the Liouville measure is finite.

 Consider the one-form $\beta$ in $SM$ defined for $\theta=(p,v)$ by 
$$\beta_{\theta}(\xi) = g_{\theta}^S(\xi,G(\theta)) = \langle D_{\theta}\pi(\xi),v \rangle_p.$$
Observe that $ker\, \beta_{\theta}\supset V(\theta) \cap T_{\theta} SM $.  It is possible to prove that a vector $\xi \in T_{\theta}TM$ lies in $T_{\theta}SM$ 
with $\theta=(p,v)$ if and only if $\langle \alpha_{\theta}(\xi) , v\rangle = 0$. Furthermore, 
$\beta$ is  a contact form invariant by the geodesic flow whose Reeb vector field is the geodesic vector field $G$. Furthermore, the sub-bundle $S=\ker \beta$
is the orthogonal complement of the subspace spanned by $G$. Since $\beta$ is invariant by the geodesic flow, then the sub-bundle $S$ is invariant by $\phi^{t}_{_{M}}$, \emph{i.e.}, $\phi^t_{_{M}}(S(\theta)) = S(\phi^{t}_{_{M}}(\theta))$ for all $\theta\in SM$ and for all $t\in\mathbb{R}$.

To understand the behavior of $d\phi^t_{_{M}}$ let us introduce the definition of a Jacobi field.  A vector field $J$ along of a geodesic
$\gamma_{\theta}$ is called the Jacobi field if it satisfies the equation 
\begin{equation}
J'' + R(\gamma'_{\theta},J)\gamma'_{\theta} = 0,
\end{equation}
where $R$ is the Riemann curvature tensor of $M$ and $``\,'\,"$ denotes the covariant derivative along $\gamma_{\theta}$. Note that, for 
$\xi=(w_1,w_2) \in T_{\theta}SM$,  with $w_1, w_2 \in T_{p} M$ and $\langle v, w_2 \rangle =0$, it is known that
$d(\phi^{t}_{_{M}})_{\theta}(\xi) = (J_{\xi}(t), J'_{\xi}(t))$, where $J_{\xi}$ denotes the unique Jacobi vector field along $\gamma_{\theta}$ such that 
$J_{\xi}(0) = w_1$ and $J'_{\xi}(0) = w_2$. For more details see \cite{gabriel:paternain}.

\subsection{No conjugate points and Ricatti equation}\label{NCP}
Suppose that $p$ and $q$ are points on a Riemannian manifold, and $\gamma$ is a geodesic that connects $p$ and $q$. Then $p$ and $q$ are \emph{conjugate points along} $\gamma$  if there exists a non-zero Jacobi field along $\gamma$  that vanishes at $p$ and $q$. 
When neither two  points in $M$ are conjugate, we say the manifold $M$ \emph{has no conjugate points}. Another important type of manifold for this paper are the manifolds without focal points, we say that a manifold  $M$ \emph{has no focal point}, if for any unit speed geodesic $\gamma$ in $M$  and  for any 
Jacobi vector field $Y$ on $\gamma$  such that $Y(0) = 0$ and $Y'(0)\neq 0$ we have $(||Y||^2)'(t)>0$, for any $t > 0$. It is clear that if
a manifold has no focal points, then also has no conjugate points. \ \\
\indent The more classical example of manifolds without focal and therefore without conjugate points  are the manifolds of non-positive curvature. It is possible to construct a manifold having  positive curvature somewhere, and without conjugate points (cf. \cite{DP:03}) . Manifolds without conjugate points are connected with manifolds with Anosov geodesic flow. In  \cite{mane:mane}, Ma\~n\'e proved that, for manifolds of finite volume, the Anosov property of the geodesic flow implies no conjugate points. This had been proved early by Klingenberg (cf.  \cite{kli:kli}) in the compact case. Recently, in \cite{RomanaMelo3}, the case of infinite volume was proved by the authors.

The last fact showed that, if we would like to work with the geodesic flows of  Anosov type, we assume then that our manifolds have no conjugate points (condition superfluous for manifolds of finite volume via Ma\~ne result). Therefore, from now on, we can assume that the manifold $M$ has no
conjugate points.\\ 
\indent Now suppose that $M$ has no conjugate points and its sectional curvatures are bounded below by $-c^2$. In this case, if the
geodesic flow $\phi^t_{_{M}}:SM \to SM$ is Anosov, then in \cite{bol:bol}, Bolton showed that $E^{s}(\theta) \cap V(\theta) = \{0\}$ and 
$E^{u}(\theta) \cap V(\theta) = \{0\}$ for all $\theta \in SM$. This last property will be very useful for the proof of Theorem \ref{curvature}.\\
\indent There are important subbundles of $TSM$ to study the dynamic behavior of the geodesic flow of manifolds without conjugate points, the \emph{Green subbundles} are defined as follows: 
$$G^s_{\theta}:=\{\xi\in T_{\theta}SM: \xi \,\, \text{is orthogonal to}\,\, G(\theta) \, \, \text{and} \, \, ||J_{\xi}(t)|| \, \, \text{is bounded to}\,\,  t\geq 0\}$$
and 
$$G^u_{\theta}:=\{\xi\in T_{\theta}SM: \xi \,\, \text{is orthogonal to}\,\, G(\theta) \, \, \text{and} \, \, ||J_{\xi}(t)|| \, \, \text{is bounded to}\,\,  t\leq 0\},$$
where $G(\theta)$ is the geodesic vector field, which are called the \emph{stable} and \emph{unstable Green subbundle}, respectively. {For a manifold without conjugate points and dimension $n$, then the dimension of Green subbundles is $n-1$. Moreover, if the curvature is nonpositive, then the Green subbundles depend continuously on $\theta\in SM$. In the general case, so characterizes Anosov geodesic flows,  if the Green subbundles depend continuously on $\theta$ and $G^s_{\theta}\cap G^u_{\theta}=\{0\}$, then the geodesic flow is Anosov (see \cite{eb:eb} for more details on Green subbundles).}

\indent For  $\theta = (p, v) \in SM$, we denote by $N(\theta):= \{w \in T_{x}M : \langle w, v\rangle = 0\}$.
By the 
identification of the Subsection \ref{GeoFlow} we can  write $S(\theta):=\emph{ker} \, \beta = N(\theta)\times N(\theta)$, $V(\theta) \cap S(\theta) =  \{0\} \times N(\theta) $ and $H(\theta) \cap S(\theta) =  N(\theta) \times \{0\} $. Thus, if  $E \subset S(\theta)$ is a subspace,
$\dim E = n-1$, and $E \cap (V(\theta) \cap S(\theta)) = \{0\}$ then 
$E \cap (H(\theta) \cap S(\theta))^{\perp} = \{0\}$. Hence, there exists a unique linear map $T: H(\theta) \cap S(\theta) \to  V(\theta) \cap S(\theta)$ such that $E$ is the graph of  $T$. In other words, there exists a unique linear map $T: N(\theta) \to N(\theta)$ such that $E = \{(v,Tv) : v \in N(\theta)\}$. Furthermore, the linear map $T$ is symmetric if and only if $E$ is Lagrangian (see \cite{gabriel:paternain}).\\

It is known that if the  geodesic flow is Anosov, then for each $\theta\in SM$, the sub-bundles $E^{s}({\theta})$ and
$E^{u}(\theta)$ are Lagrangian and $E^{s}({\theta}) \oplus E^{u}({\theta}) = S(\theta)$. Therefore, for each $t\in \mathbb{R}$, we can write 
$d(\phi^t_{_{M}})(E^{s}(\theta)) = E^{s} (\phi^{t}_{_{M}}(\theta)) = {\rm graph} \, U_{s}(t)$ and 
$d(\phi^t_{_{M}})(E^{u}(\theta)) = E^{u} (\phi^{t}_{_{M}}(\theta)) = {\rm graph} \, U_{u}(t)$, where 
$ U_{s}(t): N(\phi^{t}_{_{M}}(\theta)) \to N(\phi^{t}_{_{M}}(\theta))$ and $ U_{u}(t): N(\phi^{t}_{_{M}}(\theta)) \to N(\phi^{t}_{_{M}}(\theta))$ are symmetric maps.

Now we describe a useful method of L. Green (cf. \cite{gre:gre}), to see what properties the maps $U_{s}(t)$ and $U_{u}(t)$ satisfy.

Let $\gamma_{\theta}$ be a geodesic, and consider $V_1,\mathellipsis,V_n$ a system of parallel orthonormal vector fields  along
$\gamma_{\theta}$ with $V_n(t) = \gamma'_{\theta}(t)$.
If $Z(t)$ is a perpendicular vector field along $\gamma_{\theta}(t)$, we can write $$Z(t) =\displaystyle \sum_{i=1}^{n-1} y_{i}(t)V_i(t).$$ 
Note that $Z(s)$ can be identified with the curve $\alpha(s) = (y_{1}(s), \mathellipsis,y_{n-1}(s))$ and $Z'(s)$ can be identified with the curve $\alpha'(s) = (y_{1}'(s), \mathellipsis,y_{n-1}'(s))$. Conversely, 
any curve in $\mathbb{R}^{n-1}$ can be identified with a perpendicular vector field on $\gamma_{\theta}(t)$, so we can identify $N(\phi^{t}_{_{M}}(\theta))$ with $\mathbb{R}^{n-1}$ and consider the maps associated to stable and unstable subbundles defined in $\mathbb{R}^{n-1}$.

Now for each $t \in \mathbb{R}$, consider the symmetric matrix $R(t) = (R_{i,j}(t))$, where\break
$1 \leq i,j \leq n-1$, $R_{i,j}(t) = \langle R(\gamma'_{\theta}(t), V_i(t))\gamma'_{\theta}(t),  V_j(t)) \rangle$ and $R$ is the curvature tensor of $M$. 
The family of operators $ U_{s}(t): \mathbb{R}^{n-1} \to \mathbb{R}^{n-1}$ and $ U_{u}(t): \mathbb{R}^{n-1} \to \mathbb{R}^{n-1}$ satisfies the Ricatti
equation

\begin{eqnarray}\label{Ricatti}
U'(t) + U^{2}(t) + R(t) = 0,
\end{eqnarray}
(see the book \cite[Section 2]{gabriel:paternain}).\\
The \emph{Ricci Curvature} in a point $(x,v)\in SM$ is defined as 
$$\text{Ric}(x,v)= \frac{1}{n-1} \sum_{i=1}^{n-1} R(v, v_i, v, v_i),$$
where $\{v,v_1,v_2,\dots, v_{n-1}\}$ is an orthonormal basis of $T_{x}M$ (see \cite{docarmo} for more details). \\
Moreover, using the previous discussion, we have that the Ricci curvature along $\gamma_\theta(t)$ is given by 
\begin{equation}\label{Ricci Curvature}
\text{Ric}(\phi^t(\theta))=\frac{\text{tr}R(t)}{n-1},
\end{equation}
where $\text{tr}R(t)$ denotes trace of $R(t)$.
\ \\
Now consider the $(n-1) \times (n-1)$ matrix Jacobi equation
\begin{equation}\label{Jacobi}
Y''(t) + R(t)Y(t) = 0.
\end{equation}

If $Y(t)$ is a solution of (\ref{Jacobi}) then for each $x \in \mathbb{R}^{n-1}$, the curve $ \beta(t) = Y(t)x$ corresponds to a Jacobi perpendicular vector
on $\gamma_{\theta}(t)$.  For $\theta\in SM$, $s\in \mathbb{R}$, we consider  $Y_{\theta,s}(t)$ to be the unique solution of ($\ref{Jacobi}$) satisfying
$Y_{\theta,s}(0) = I$ and $Y_{\theta,s}(s) = 0$. In  \cite{gre:gre}, Green proves  that $\ds\lim_{s \to -\infty}Y_{\theta,s}(t)$ exists for all $\theta\in SM$ (see also \cite[Sect. 2]{eb:eb}). He also shows that if we define:
\begin{equation}\label{Jacobi Tensor}
Y^{+}_{\theta}(t) = \lim_{s \to -\infty}Y_{\theta,s}(t)
\end{equation}

\noindent we obtain a solution of Jacobi equation (\ref{Jacobi}) such that $\det Y^{+}_{\theta}(t)\neq 0$. Moreover, it is proved in \cite{gre:gre} 
(see also  \cite{mane:freire} and \cite{eb:eb}) that $\ds\frac{DY^{+}_{\theta}}{Dt}(t)=\lim_{s\to -\infty}\frac{DY_{\theta,s}}{Dt}(t)$. However, if 
$$U_s(\theta)=\frac{DY_{\theta,s}}{Dt}(0) ; \, \, U^{+}(\theta)=\frac{DY^{+}_{\theta}}{Dt}(0),$$
then  $$U^{+}(\theta)=\lim_{s\to -\infty}U_{s}(\theta).$$
And it can be proved easily that (see \cite{mane:freire})
$$U^{+}(\phi^{t}_{_{M}}(\theta))=\frac{DY^{+}_{\theta}}{Dt}(t){Y^{+}_{\theta}}^{-1}(t)$$
for every $t \in \mathbb{R}$. It follows that $U^{+}$ is a symmetric solution of the Ricatti equation (\ref{Ricatti}).
Analogously, taking the limit when $ s\to +\infty$,  we have defined $U^{-}(\theta)$, which also satisfies the Ricatti equation (\ref{Ricatti}). Furthermore, in \cite{gre:gre}, Green also showed that in the case of curvature bounded below by $-c^2$, symmetric solutions of the Ricatti equation which are defined for all $t\in \mathbb{R}$ are bounded by $c$. In particular,
$$\sup_{t\in \mathbb{R}}||U^{+}(t)||\leq c.$$\\
When the geodesic flow is Anosov, $U_u = U^{+}$ and $U_s=U^-$ with
$$\max\Big\{ \sup_{t\in \mathbb{R}}||U^{s}(t)||,\, \,  \sup_{t\in \mathbb{R}}||U^{u}(t)||\Big\} \leq c.$$


\section{Proof of Theorem \ref{curvature}}\label{Proof}

In this section, we show several lemmas that will be used to obtain the proof of Theorem \ref{curvature}. Our first lemma (Lemma 3.1) proves the first part of Theorem \ref{curvature} in a more general setting using only the no conjugate points condition (or only the Anosov condition by \cite{RomanaMelo3}), without any condition about the volume. 
Since any manifold of finite volume with Anosov geodesic flow has no conjugate points (cf. \cite{mane:mane}), Lemma 3.1 implies in fact the proof of the of item \text{(a)} of Theorem \ref{curvature}.  
\begin{lemma}\label{inequality curvature}
Let $M$ be a complete manifold with curvature bounded below by $-c^2$ without conjugate points and whose geodesic flow is Anosov. 	
If the constant of contraction of the geodesic flow is $\lambda$ then $\lambda \geq e^{-c}$.
\end{lemma}
\begin{proof}
Since the curvature is bounded below by $-c^2$, then by Lemma 2.16 at \cite{handbook:handbook} (see also \cite[Lemma 2.17]{handbook:handbook}), we have  
\begin{equation}\label{EQNEW1}
\frac{1}{\sqrt{1+c^2}}e^{-ct}||\xi||\leq ||d(\phi^t_{_{M}})_{\theta}(\xi)||, \, \,   \,\, t\geq 0, \,\ \, \, \xi \in {E^s(\theta)},
\end{equation}
and 
\begin{equation}\label{EQNEW2}
||d(\phi^t_{_{M}})_{\theta}(\eta)||\leq e^{ct}||\eta||\sqrt{1+c^2}, \, \,  \, \,  t\geq 0, \, \,\, \eta \in {E^u(\theta)}.
\end{equation}
Using the definition of Anosov flow in (\ref{EQNEW1}) (similar argument using  (\ref{EQNEW2})), we have 
$$\frac{1}{\sqrt{1+c^2}}e^{-ct}||\xi||\leq ||d(\phi^t_{_{M}})_{\theta}(\xi)||\leq C\lambda^{t}||(\xi)||\, \,  \text{for all} \, \, \xi \in {E^s(\theta)}.$$
If we write $\lambda=e^{-\kappa}$, the last inequality becomes 
$$e^{-(c-\kappa)t}\leq C\sqrt{1+c^2}, \, \,   \,\, t\geq 0,$$
which is valid if and only if $c\geq \kappa$ or equivalently $\lambda=e^{-\kappa}\geq e^{-c}$ as we desired.

\end{proof}

\begin{corollary}\label{NonNegCur}
No manifold $M$ of finite volume and non-negative curvature has the geodesic flow of the Anosov type. 
\end{corollary}
\begin{proof}
By contradiction, if such manifold $M$ exists, then $M$ has no conjugate points (cf. \cite{mane:mane}). Since the curvature $K$ of $M$ is non-negative, then for all $\epsilon>0$, $K\geq -\epsilon^{2}$. Let $0<\lambda<1$ be the constant of contraction of $\phi^t_{_{M}}$, then by Theorem \ref{curvature} we have $\lambda\geq e^{-\epsilon}$ for all $\epsilon>0$, which implies that $\lambda\geq 1$ and this is a contradiction.   
\end{proof}

\subsection{Rigidity and Lyapunov exponents}

In this subsection, we shall prove that the map  $U^{+}(\theta)$ (from Subsection \ref{NCP}) has all information about Lyapunov exponents. The following 
Lemma was proved by Freire and Ma\~n\'e (\cite[Theorem II]{mane:freire}) in compact manifolds without conjugate points. Here we do the proof in the non-compact case for Anosov geodesic flows.
\begin{lemma}\label{Mane for Anosov}
Let $M$ be a complete manifold with curvature bounded below by $-c^2$ without conjugate points and whose geodesic flow is Anosov. Then 
$$\ds\lim_{t\to +\infty}\frac{1}{t}\log|\det d(\phi^{t}_{_{M}})_{\theta}|_{E^{u}(\theta)}|=\lim_{t\to +\infty}\frac{1}{t}\int_{0}^{t}\emph{tr}(U^{+}(\phi^{s}_{_{M}}(\theta)))\, ds$$
holds for every $\theta\in SM$.
\end{lemma}
\begin{proof}
For each $\theta=(p,v)\in SM$, we denote by $N(\theta)$ the subspace of $T_{p}M$  orthogonal to $v$. Then by construction,
for all $x\in N(\theta)$ the Jacobi field $Y^{+}_{\theta}(t)x$ is an unstable Jacobi field. Thus, it follows that  $U^{+}(\theta)$ satisfies that 
$$E^{u}(\theta)=\text{graph}\, U^{+}(\theta)=\{(x,U^{+}(\theta)x):x\in N(\theta)\}.$$
Let $\pi_{\theta}\colon E^{u}(\theta)\to N(\theta)$ the projection in the first coordinate. Then 
$$\pi_{\theta}^{-1}(v)=(v, U^{+}(\theta)v).$$
Therefore
\begin{equation}\label{projections}
\sup_{\theta\in SM}||\pi^{-1}_{\theta}||\leq \sqrt{1+c^2}.
\end{equation}
Moreover, for $(v,w)\in E^{u}(\theta)$
\begin{eqnarray}
\pi^{-1}_{\phi^{t}_{_{M}}(\theta)}\circ Y^{+}_{\theta}(t)\circ \pi_{\theta}(v,w)&=& \pi^{-1}_{\phi^{t}_{_{M}}(\theta)}( Y^{+}_{\theta}(t)v)\vphantom{\,}\nonumber \\
&=&(Y^{+}_{\theta}(t)v, U^{+}(\phi^{t}_{_{M}}(\theta))Y^{+}_{\theta}(t)v) \nonumber \\
&=&\left( Y^{+}_{\theta}(t)v, \frac{DY^{+}_{\theta}}{Dt}(t)v\right) \nonumber\\
&=&d\phi^{t}_{_{M}}(v,w).\nonumber
\end{eqnarray}
In other words, 
\begin{equation}\label{diferential and projections}
d(\phi^{t}_{_{M}})_{\theta}|_{E^{u}_{\theta}}=\pi^{-1}_{\phi^{t}_{_{M}}(\theta)}\circ Y^{+}_{\theta}(t)\circ \pi_{\theta}.
\end{equation}
Thus, by the equation (\ref{diferential and projections})
\begin{eqnarray}\label{conjugation}
\ds\lim_{t\to +\infty}\frac{1}{t}\log|\det\,  d(\phi^{t}_{_{M}})_{\theta}|_{E^{u}(\theta)}|&=&\lim_{t\to +\infty}\frac{1}{t}\log |\det  \pi^{-1}_{\phi^{t}_{_{M}}(\theta)} |+\lim_{t\to +\infty}\frac{1}{t}\log |\det  Y^{+}_{\theta}(t)|\nonumber \\
&& + \lim_{t\to +\infty}\frac{1}{t}\log |\det \pi_{\theta}|.
\end{eqnarray}
From Linear Algebra we know 
\begin{equation}\label{Linear Algebra}
\frac{1}{({|\det A^{-1}|)^{\frac{1}{m}}}}\leq ||A||\leq \frac{m||A^{-1}||^{m-1}}{|\det A^{-1}|},
\end{equation}
where $A$ is an invertible  linear map of  vector spaces of dimension $m$.\

As \rm{dim}\,$E^{u}(\theta)=n-1$, using the inequality (\ref{Linear Algebra}) and the inequalities
$||\pi_{\theta}||\leq 1$, \break $1\leq ||\pi^{-1}_{\theta}||\leq \sqrt{1+c^{2}}$ for all $\theta$, we concluded that 
$$1\leq ||\pi^{-1}_{\phi^{t}_{_{M}}(\theta)}||\leq (n-1)||\pi_{\phi^{t}_{_{M}}(\theta)}||^{n-2}|\det  \pi^{-1}_{\phi^{t}_{_{M}}(\theta)}|\leq (n-1)|\det  \pi^{-1}_{\phi^{t}_{_{M}}(\theta)}|$$
and 
$$|\det  \pi^{-1}_{\phi^{t}_{_{M}}(\theta)}|^{1-\frac{1}{n-1}}\leq (n-1)||\pi^{-1}_{\phi^{t}_{_{M}}(\theta)}||^{n-2}\leq (n-1)({1+c^{2}})^{\frac{n-2}{2}},$$
that is,
\begin{equation}\label{bounded of determinant}
\frac{1}{n-1}\leq |\det  \pi^{-1}_{\phi^{t}_{_{M}}(\theta)}|\leq (n-1)^{\frac{n-1}{n-2}}({1+c^{2}})^{\frac{n-1}{2}}.
\end{equation}
Therefore, the inequalities (\ref{conjugation}) and (\ref{bounded of determinant}) provide us with
\begin{equation}\label{bd1}
\ds\lim_{t\to +\infty}\frac{1}{t}\log|\det d(\phi^{t}_{_{M}})_{\theta}|_{E^{u}(\theta)}|=\lim_{t\to +\infty}\frac{1}{t}\log |\det  Y^{+}_{\theta}(t)|.
\end{equation}
Since $\det  Y^{+}_{\theta}(t)\neq 0$, then is easy to prove that 
\begin{equation}\label{bd2*}
\frac{d}{dt} \det Y^{+}_{\theta}(t)=\det Y^{+}_{\theta}(t)\, \text{tr}\left( \frac{DY^{+}_{\theta}(t)}{Dt}{Y^{+}_{\theta}}^{-1}(t)\right).
\end{equation}
From (\ref{bd1}) and (\ref{bd2*})
\begin{eqnarray}\label{bd2}
\ds\lim_{t\to +\infty}\frac{1}{t}\log|\det d(\phi^{t}_{_{M}})_{\theta}|_{E^{u}(\theta)}|&=&\lim_{t\to +\infty}\frac{1}{t}\int_{0}^{t}\text{tr}\left( \frac{DY^{+}_{\theta}(s)}{Ds}{Y^{+}_{\theta}}^{-1}(s)\right)ds \nonumber \\
&=&\lim_{t\to +\infty}\frac{1}{t}\int_{0}^{t}\text{tr}\left(U^{+}(\phi^{s}_{_{M}}(\theta)\right)\,  ds.
\end{eqnarray}
\end{proof}
\noindent The following lemma shows the rigidity of Lyapunov exponent. 
\begin{lemma}[Rigidity of Lyapunov Exponent]\label{R. H. D.} In the same conditions as \rm{Theorem} \ref{curvature}. \textit{If} $\lambda=e^{-c}$, \textit{then}  
\begin{equation}
\lim_{t \to +\infty}\frac{1}{t}\log ||d(\phi^{t}_{_{M}})_{\theta}(\xi)||=-c \ \ \ \text{and} \ \ \lim_{t \to +\infty}\frac{1}{t}\log ||d(\phi^{t}_{_{M}})_{\theta}(\eta)||=c
\end{equation}
for all $\theta\in SM$, $\xi\in E^{s}({\theta})$ and $\eta\in E^{u}({\theta})$.
\end{lemma}
\begin{proof}

If $\lambda = e^{-c}$, then by definition of Anosov flow, (\ref{EQNEW1}), and (\ref{EQNEW2}) we have 

$$\frac{1}{\sqrt{1+c^2}}e^{-ct}||\xi||\leq ||d(\phi^t_{_{M}})_{\theta}(\xi)||\leq Ce^{-ct}||\xi||, \, \, \,\,\, t\geq 0, \,\,\, \, \, \xi \in {E^s(\theta)},$$
and 

$$\frac{1}{C}e^{ct}||\eta||\leq ||d(\phi^t_{_{M}})_{\theta}(\eta)||\leq e^{ct}||\eta||\sqrt{1+c^2}, \, \,\, \, \,\,\, t\geq 0, \,\,\, \, \,\, \, \eta \in {E^u(\theta)}.$$
The above inequalities conclude the proof of Lemma.
\end{proof}
\begin{corollary}\label{C. R. H. D.}
If $\lambda=e^{-c}$, then 
$$
\ds\lim_{t \to +\infty}\frac{1}{t}\log |\det \,d(\phi^{t}_{_{M}})_{\theta}|_{E^{s}(\theta)}|=-c\cdot \,\emph{dim}\, E^{s}(\theta)=-c(n-1),
$$ and 
$$\ds\lim_{t \to +\infty}\frac{1}{t}\log |\det d(\phi^{t}_{_{M}})_{\theta}|_{E^{u}(\theta)}|=c\cdot \, \emph{dim}\, E^{u}(\theta)=c(n-1).$$
\end{corollary}
\begin{proof}
From Lemma \ref{R. H. D.}, we have  $\ds\lim_{t \to +\infty}\frac{1}{t}\log ||d(\phi^{t}_{_{M}})_{\theta}(\eta)||=c$ for all $\theta\in SM$ and\break
$\eta\in E^{u}(\theta)$, then it is easy to see that 
$$\ds\lim_{t \to +\infty}\frac{1}{t}\log |\det d(\phi^{t}_{_{M}})_{\theta}|_{E^{u}(\theta)}|=c\cdot \text{dim} E^{u}(\theta)=c(n-1).$$

The proof for $E^{s}(\theta)$ is analogous.
\end{proof}

\noindent To conclude this section, we prove the following Lemma, which provides the final step in proving the item \text{(b)} of the Theorem \ref{curvature}.
\begin{lemma}[Rigidity]\label{L. R. H. D.}
In the same conditions as \emph{Theorem \ref{curvature}},  
$\lambda=e^{-c}$ if and only if  the sectional curvature of the manifold $M$ is constant, equal to $-c^2$.
\end{lemma}
\begin{proof}
Note that if $K=-c^{2}$, then the classical proof that the geodesic flow is Anosov implies that  $\lambda=e^{-c}$. So our task is to prove the other side. 
In fact: from  Lemma \ref{Mane for Anosov} and Corollary \ref{C. R. H. D.}, we have that 
$$\lim_{t\to +\infty}\frac{1}{t}\int_{0}^{t}\text{tr}(U^{+}(\phi^{s}_{_{M}}(\theta)))\, 
ds=\ds\lim_{t\to +\infty}\frac{1}{t}\log|\det d(\phi^{t}_{_{M}})_{\theta}|_{E^{u}(\theta)}|=c(n-1),$$
for all $\theta\in SM$.

Since $U^{+}(\phi^{s}_{_{M}}(\theta))$ is symmetric then is easy to see that 
$$(\text{tr}(U^{+}(\phi^{s}_{_{M}}(\theta))))^{2}\leq (n-1)\text{tr}(({U^{+}}(\phi^{s}_{_{M}}(\theta))^{2})$$
Since the sectional curvature satisfies $K\geq -c^{2}$, then taking  trace and integrating from $0$ to $t$  the Ricatti equation (\ref{Ricatti}) and taking into account (\ref{Ricci Curvature}),  from 
the Cauchy-Schwarz inequality we have  
\begin{eqnarray}\label{eq6 R. H. D.}
\frac{1}{t}\int_{0}^{t}\text{tr}(U^{+}(\phi^{s}_{M}(\theta)))ds & \leq & \sqrt{\frac{1}{t}\int_{0}^{t}(\text{tr}(U^{+}(\phi^{s}_{_{M}}(\theta))))^{2}}ds \nonumber \\ 
&\leq & \sqrt{\frac{n-1}{t}\int_{0}^{t}{\text{tr}((U^{+}}(\phi^{s}_{_{M}}(\theta)))^{2})}ds\nonumber \\ 
&=&\sqrt{-\frac{n-1}{t}\left( \int_{0}^{t}\text{tr}((U^{+}){'}(\phi^{s}_{_{M}}(\theta)))ds+\int_{0}^{t}\text{tr}(R(s))ds)\right) }\nonumber \\ 
&=&\sqrt{-\frac{n-1}{t}\left(\text{tr}((U^{+})(\phi^{t}_{_{M}}(\theta)))-\text{tr}((U^{+})(\theta))\right)-\frac{(n-1)^{2}}{t}\int_{0}^{t}\text{Ric}(\phi^{s}_{_{M}}(\theta)))ds} \nonumber \\ 
&\leq & \sqrt{-\frac{n-1}{t}\left(\text{tr}((U^{+})(\phi^{t}_{_{M}}(\theta)))-\text{tr}((U^{+})(\theta))\right)+(n-1)^{2}c^2}.
\end{eqnarray}
Hence, since $||U^{+}(\phi^{t}_{_{M}}(\theta))||\leq c $, we have 
$\ds \lim_{t\to +\infty}-\frac{n-1}{t}\left(\text{tr}((U^{+})(\phi^{t}_{_{M}}(\theta)))-\text{tr}((U^{+})(\theta))\right)=0$. Thus, taking limit as
$t\to +\infty$ in the inequality  (\ref{eq6 R. H. D.}), we get   
\begin{equation}\label{eq1 LRHD}
\ds \lim_{t\to +\infty}\frac{1}{t}\int_{0}^{t}\text{Ric}(\phi^{s}_{_{M}}(\theta)))ds=-c^2.
\end{equation}
To conclude our argument, let us use the Birkhoff Ergodic theorem. First, we note that as $K\geq -c^2$,  then  the negative part of the Ricci curvature
is integrable on $SM$. Thus, from a result of  Guimar\~aes in  \cite{flo:gui}, we have that the Ricci curvature is integrable on $SM$.  Therefore, as $M$ has finite
volume, from Birkhoff ergodic theorem and equation (\ref{eq1 LRHD}) we have    
$$\int_{SM}\text{Ric}(\theta)d\mu=-c^2.$$
Since the sectional curvature satisfies $K\geq -c^{2}$, then the previous equation implies that $\rm{Ric}(\theta)\equiv -c^2$. Thus, we conclude
that $$K\equiv-c^2.$$
\end{proof}
As an immediate consequence of Lemma \ref{L. R. H. D.} we get the following corollary, which proves that  Conjecture 2 implies Conjecture 1.
\begin{corollary}\label{Conj2-Conj1}
Let $M$ be a complete Riemannian manifold with finite volume, sectional curvature bounded below, and Anosov geodesic flow. If the unstable Lyapunov exponents of all periodic orbits are constants equal to $c_{{\inf}}$, then $M$ has constant negative sectional curvature equal to $-c_{_{\inf}}^2$.
\end{corollary} 

\section{Rigidity of Conjugacy}\label{proof of T-R-C}
The main purpose of this section is to prove Theorem \ref{R-conjugation}, Theorem \ref{R-conjugation-WCP}, and its corollaries. The strategy of the proof of Theorem \ref{R-conjugation} is to show that, if $\lambda_N$  is a constant of contraction of $\phi^t_N$ and $\lambda_M$ is a constant of contraction of $\phi^t_{M}$, then ${\lambda}_{_M}={\lambda}_{_N}$ and apply  the Theorem \ref{curvature} to conclude our result. For
this sake, we need the following lemma. 
\begin{lemma}\label{derivative distance}
Let $(S,g)$ be a Riemannian manifold and $d$ its Riemannian distance. If $\alpha(t)$ is a curve on $S$, differentiable at $t=0$, then 
\begin{equation*}\label{equation derivative distance}
\ds \lim_{t\to 0^+}\frac{d(\alpha(t),\alpha(0))}{t}=||\alpha'(0)||_{g}
\end{equation*}
\end{lemma}
\begin{proof}
The exponential map $\text{exp}_{\alpha(0)}$ defines a diffeomorphism in a neighborhood of $0\in $ ${T_{\alpha(0)}S}$ on a neighborhood of $\alpha(0)$
with $D\left(\text{exp}_{\alpha(0)}\right)_{0}=\text{Id}$. Then, for a small $t$  we have that 
$$d(\alpha(t), \alpha(0)))=\Big\Vert\text{exp}^{-1}_{\alpha(0)}(\alpha(t))\Big\Vert_{g}.$$
Thus, 
\begin{eqnarray}\label{e4444}
\lim_{{t}\to 0^+}\dfrac{d(\alpha(t), \alpha(0))}{{t}}&=&\Big\Vert\frac{d}{dt}\Big|_{t=0}\text{exp}^{-1}_{\alpha(0)}\alpha(t)\Big\Vert_{g} \
=\Big \Vert\underbrace{D\left(\text{exp}_{\alpha(0)}^{-1}\right)_{\alpha(0)}}_{\text{Id}}(\alpha'(0))\Big \Vert_{g} \nonumber 
= \Big \Vert \alpha'(0)\Big \Vert_{g}.
\end{eqnarray}
\end{proof}

\begin{proof}[\textbf{Proof of Theorem \ref{R-conjugation}}]
We note that as $M$ has finite volume and $\phi^t_{_{M}}$ is Anosov, then  Corollary \ref{NonNegCur} implies that $\inf K_{M}<0$, thus $\sup K_N<0$ and $\phi^t_{_{N}}$ is also Anosov.  We put $-a^2:=\sup K_N$.\\
Now assume that $f:SM \to SN$ is a $1$-equivalence between $\phi^t_{_{M}}$ and $\phi^t_{_{N}}$ with constant $C_1$ and $C_2$ (see (\ref{e1111})).
Without loss of generality, we denote by $d$ the metric of $SM$ and $SN$.
Denote by $\Gamma$ the set of points in $SM$ where there exists $df_{\theta}$, then as $f$ is a Lipschitz map, the set $\Gamma$ has full \emph{Liouville} measure.  Consider $\theta\in \Gamma$, then as $f$ is bi-Lipschitz with the help of Lemma \ref{derivative distance} we have 
$$C_1||v||\leq ||df_{\theta}(v)||\leq C_2||v||, \,\,\text{for all}\,\, v\in T_{\theta}SM.$$
Consequently, $df_{\theta}$ is an injective linear map. So we can define the subspace $F^{s(u)}_{f(\theta)}$ satisfying the equation 
$$F^{s(u)}_{f(\theta)}:=df_{\theta}(E^{s(u)}_{\theta}),$$
where $E^{s(u)}$ are the stable and unstable subbundle from the definition of Anosov flow of $\phi^t_{_{M}}$.
It is easy to see that $\Gamma$ is $\phi^t_{M}$-invariant and since
$$df_{\phi_{M}^{t}(\theta)}=d(\phi^t_{N})_{f(\theta)}\circ df_{\theta}\circ d(\phi^t_{M})_{\theta}^{-1},$$
 the subspaces $F^{s(u)}_{f(\theta)}$ are invariant by $d\phi^t_{N}$.
\\
\ \\
\indent \textbf{Claim:} For all $\theta\in \Gamma$ holds that $$F_{f(\theta)}^{s(u)}\subset E^{s(u)}_{f(\theta)},$$
where $E^{s(u)}_{f(\theta)}$ from definition of Anosov flow of $\phi^t_{_{N}}$.
\noindent  \begin{proof}[\textbf{Proof of Claim}]
 Since $f$ is a $1$-equivalence, then 
 \begin{equation}\label{eq-C01*}
C_1\cdot d(\phi^t_{M}(x),\phi^t_{M}(y))\leq d(\phi^{t}_{N}(f(x)),\phi^t_{N}(f(y)))\leq C_2\cdot d(\phi^t_{M}(x),\phi^t_{M}(y)),
\end{equation} 
for all $t\in \mathbb{R}$ and all $x,y\in SM$.\\
Consider $\theta\in \Gamma$ and 
   $\xi\in E^s_{\theta}\setminus \{0\}$   and  $\beta(r)\subset SM$ a $C^1$-curve, such that $\beta(0)=\theta$ and $\beta'(0)=\xi$. From 
Lemma \ref{derivative distance}, for $t\in \mathbb{R}$ we have 
$$\lim_{h\to 0}\frac{d(\phi^{t}_{M}(\beta(h), \phi^{t}_{M}(\beta(0)))}{h}=||d(\phi^t_{M})_{\theta}(\xi)||$$
and 
$$\lim_{h\to 0}\frac{d(\phi^{t}_{N}(f(\beta(h)), \phi^{t}_{N}(f(\beta(0))))}{h}=||d(\phi^t_{N})_{f(\theta)}(df_{\theta}\xi)||.$$
The last two equalities and inequality (\ref{eq-C01*})  implies that 
\begin{equation*}
C_1\cdot ||d(\phi^t_{M})_{\theta}(\xi)||\leq ||d(\phi^t_{N})_{f(\theta)}(df_{\theta}(\xi))||\leq C_2\cdot ||d(\phi^t_{M})_{\theta}(\xi)|| 
\end{equation*}
or 
\begin{equation}\label{eq3-C00001*}
C_1\cdot \frac{||d(\phi^t_{M})_{\theta}(\xi)||}{||\xi||}\leq \frac{||d(\phi^t_{N})_{f(\theta)}(df_{\theta}(\xi))||}{||df_{\theta}(\xi)||}\cdot \frac{||df_{\theta}(\xi)||}{||\xi||}\leq C_2\cdot \frac{||d(\phi^t_{M})_{\theta}(\xi)||}{||\xi||}.
\end{equation}
Consequently, for all $t\in \mathbb{R}$

\begin{equation}\label{eq3-CO1*}
\frac{C_1}{C_2}\cdot ||d\phi^t_{M}|_{E^{s}_{\theta}}||\leq ||d\phi^t_{N}|_{F^{s}_{f(\theta)}}||\leq \frac{C_2}{C_1}\cdot ||d\phi^t_{M}|_{E^{s}_{\theta}}||.
\end{equation}
\noindent Similar inequality holds for the stable space $F^s_{F(\theta)}$.\\
Now we notice that
$$E^{s}_{f(\theta)}=\{\xi\in T_{\theta}SM: \lim_{t\to\infty}||d\phi^t_{N}(\xi) ||=0\}.$$
Thus, the last characterization and (\ref{eq3-CO1*}) allows us to complete the proof of claim in the stable case. The proof of the unstable case is completely analogous.
\end{proof}
Since a constant of contraction of $\phi^{t}_{_{N}}$ is   $e^{-{\sqrt{-\sup K_N}}}=e^{-a}$ (cf. \cite{handbook:handbook}), then  the inequality (\ref{eq3-CO1*}) and Claim provide

$$||d\phi^t_{M}|_{E^{s}_{\theta}}||\leq \frac{C_2}{C_1} e^{-a t} \ \ \text{and} \ \ ||d\phi^{-t}_{M}|_{E^{u}_{\theta}}||\leq \frac{C_2}{C_1} e^{-at },$$
for all $t\geq 0$ and all $\theta\in \Gamma$. However, since $\Gamma$ has full measure, then the last inequality holds for all $\theta\in SM$.
In other words, $\lambda_M=e^{-a}$. However, by hypotheses $K_M\geq \inf K_M\geq \sup K_M = -a^2$, then Theorem \ref{curvature} implies that $K_M\equiv -a^2$, thus we conclude the proof of the theorem.
\end{proof}

\begin{proof}[\textbf{Proof of Corollary \ref{C0}}]The strategy here is to show that $C^1$-equivalences that are injective immersions are actually  $1$-equivalences, at least for the compact case.
Let $f\colon SM\to SN$ be a $C^1$ equivalence, which is a injective immersion. Then $f\colon SM \to f(SM)$ is a diffeomorphism and $f(SM)$ is a compact submanifold of $SN$. 
If we consider $f(SM)$ with the extrinsic distance of $SN$ and let $d_{f}$ be  such distance, it is easy to see that for all $\theta, \tilde{\theta}\in SM$
\begin{equation}\label{EQ1NEWNEW}
d(\theta, \tilde{\theta})\leq \sup_{\theta\in SM}||Df^{-1}_{f(\theta)}||\cdot d_{f}(f(\theta), f(\tilde{\theta})) \,\,\, \,\text{and}\,\,\,\, d(f(\theta),f(\tilde{\theta}))\leq \sup_{\theta\in SM}|Df_{\theta}|\cdot d(\theta, \tilde{\theta}).
\end{equation}
From (\ref{EQ1NEWNEW}) and Corollary \ref{COMP DISTANCE}, there is $\Gamma>1$ such that 
$$\dfrac{1}{\Gamma\cdot \displaystyle\sup_{\theta\in SM}||Df^{-1}_{f(\theta)}||}d(\theta, \tilde{\theta})\leq d(f(\theta), f(\tilde{\theta}))\leq \sup_{\theta\in SM}|Df_{\theta}|\cdot d(\theta, \tilde{\theta}).$$

Thus,  $f$ is a $1$-equivalence. Thus, Theorem \ref{R-conjugation} allows us to conclude that
$K_{M}\equiv \sup K_{N}$ as we wish.  
\end{proof}

\begin{proof}[\textbf{Proof of Theorem \ref{R-conjugation-WCP}}]
Observe that since  $M$ has no conjugate points, then $\inf K_M\leq 0$ (cf. \cite{gre:gre} and \cite{hopf:hopf}). Therefore the hypotheses on curvatures  imply that $\sup K_{N}\leq \inf K_{M}\leq 0$.  To prove the theorem, we consider two cases:
\begin{enumerate}
\item[\textbf{Case 1:}] $\sup K_{N}<0$;
\item[\textbf{Case 2:}] $\sup K_{N}=0$.
\end{enumerate}
The proof in each case is slightly different because in case 1 the behavior of the geodesic flow is hyperbolic and case 2 is not necessarily hyperbolic. \\
\ \\
\textbf{Case 1:} In this case, the proof is quite similar to the proof Theorem \ref{R-conjugation}, but we will use the hyperbolic splitting of $\phi^t_{_{N}}$ to construct a hyperbolic splitting for $\phi^t_{_{M}}$. Put $-a^2:=\sup K_{N}<0$, then  $\phi^t_{_{N}}$ is Anosov and denoted by $E^{s(u)}$ the stable and unstable bundle, respectively.
Let $f$ be a $1$-conjugacy and denote by $\Gamma$ the set of points in $SM$ where there exists $df_{\theta}$, then as $f$ is a Lipschitz map, the set $\Gamma$ has full \emph{Liouville} measure.  Consider $\theta\in \Gamma$ and note that $df_{\theta}$ is an isomorphism (\text{dim}\,$M$=\text{dim}\,$N$) and define the subspace $F^{s(u)}_{\theta}$ satisfying the equation 
$$df_{\theta}(F^{s(u)}_{\theta})=E^{s(u)}_{f(\theta)},$$
which satisfies 
\begin{equation}\label{eq-C001}
T_{\phi_{M}^{t}(\theta)}(SM) = F^s_{\phi_{M}^{t}(\theta)} \oplus \langle \phi_{M} \rangle \oplus F^u_{\phi_{M}^{t}(\theta)},
\end{equation}
for all $t\in \mathbb{R}$.\\
Similar arguments to the proof of Theorem \ref{R-conjugation}, we have that $\Gamma$ is invariant, the sub-bundles $F^{s(u)}$ are also invariant by $d\phi^t_{N}$ and satisfies
\begin{equation}\label{eq3-CO1}
\frac{C_1}{C_2}\cdot ||d\phi^t_{M}|_{F^{s(u)}_{\theta}}||\leq ||d\phi^t_{N}|_{E^{s(u)}_{f(\theta)}}||\leq \frac{C_2}{C_1}\cdot ||d\phi^t_{M}|_{F^{s(u)}_{\theta}}||,
\end{equation}
for all $t\in \mathbb{R}$.


Since the constant of contraction of $\phi^{t}_{_{N}}$ is $e^{-{\sqrt{-\sup K_N}}}=e^{-a}$ (cf. \cite{handbook:handbook}), then the inequality (\ref{eq3-CO1}) provides that for all $t\geq 0$

$$||d\phi^t_{M}|_{F^{s}_{\theta}}||\leq \frac{C_2}{C_1} e^{-a t} \ \ \text{and} \ \ ||d\phi^{-t}_{M}|_{F^{u}_{\theta}}||\leq \frac{C_2}{C_1} e^{-at }.$$
The two last inequalities together with the equation (\ref{eq-C001}) provide a hyperbolic behavior of $\phi_{_{M}}^{t}$ along $\phi^t_{_{M}}(\theta)$. Therefore, as $M$ has no conjugate points and $K_M\geq \sup K_{N}=-a^2$, then the same arguments of proof of Theorem \ref{curvature} provide that 
\begin{equation}\label{eq2-C001}
\ds \lim_{t\to +\infty}\frac{1}{t}\int_{0}^{t}\text{Ric}(\phi^{s}_{_{M}}(\theta)))ds=-a^2
\end{equation}
for all $\theta \in \Gamma$, which is  a set of full \emph{Liouville} measure. 
Then, the Birkhoff ergodic theorem and equation (\ref{eq2-C001}) give us   
$$\int_{SM}\text{Ric} \, d\mu=-a^2,$$
where $\rm{Ric}$ is the Ricci curvature.
Since the sectional curvature satisfies $K_{M}\geq -a^2$, then the previous equation implies that $\text{Ric}(\theta)\equiv -a^2$ for \emph{Liouville} almost every point $\theta\in SM$. Thus, we conclude
that $K_M\equiv-a^2$ and therefore the splitting, given by equation (\ref{eq-C001}), coincides with its hyperbolic splitting. 
To conclude the proof of case 1, we need only to prove that $K_N\equiv -a^2$. For this sake, note that since  $K_M\equiv -a^2 $ and for $\xi\in F^{u}_{\theta}$ (cf. \cite{handbook:handbook}) we have that 
\begin{equation*}
||(d\phi^t_{M})_{\theta}(\xi)||=\sqrt{1+a^2}\cdot e^{at}||\pi_{1}(\xi)||,
\end{equation*}
where $\pi_{1}(\cdot)$ is the projection on the first coordinate in the horizontal and vertical decomposition of $TSM$ (see Section \ref{GeoFlow}).\\
Thus, 
$$\displaystyle\lim_{t\to +\infty}\frac{1}{t}\log ||(d\phi^t_{M})_{\theta}(\xi)||=a$$
uniformly in bounded regions of pair $(\theta, \xi)$ with $\theta\in SM$ and $\xi\in F^u_{\theta}$. \\
Put $\Lambda=f(\Gamma)$, as $f$ is a Lipschitz map and $\Gamma$ has full \emph{Liouville} measure, then $\Lambda$ has full \emph{Liouville} measure.
As (\ref{eq3-C00001*}) is valid, we have 
\begin{equation}\label{eq2'-00001}
\lim_{t\to +\infty}\frac{1}{t}\log ||(d\phi^t_{N})_{w}(\eta)||=a,
\end{equation}
uniformly in bounded regions of pair $(w, \eta)$ with $w\in\Lambda$ and $\eta\in E^u_{w}$.\\

\indent \textbf{Claim:} For all $w\in SN$ and $\eta\in E^u_{w}$ hold that
\begin{equation*}
\lim_{t\to +\infty}\frac{1}{t}\log ||(d\phi^t_{N})_{w}(\eta)||=a.
\end{equation*}
\noindent \begin{proof}[\textbf{Proof of Claim}]
Let $w\in SN$ and $\eta\in E^u_{w}$, then by the density of $\Lambda$ and continuity of unstable bundle $E^u$, we have that there are $w_n\in \Lambda$ and $\eta_{n}\in E^u_{w_{_n}}$ such that $(w_n,\eta_n)$ converges to $(w,\eta)$.
Note that for all $t\in \mathbb{R}$ 
\begin{equation}\label{eq2''-00001}
\Big\vert a - \frac{1}{t}\log ||(d\phi^t_{N})_{w}(\eta)|| \Big\vert\leq \Big\vert a - \frac{1}{t}\log ||(d\phi^t_{N})_{w_n}(\eta_n)\Big\vert + \Big\vert \frac{1}{t}\log \frac{||(d\phi^t_{N})_{w_n}(\eta_n)||}{||(d\phi^t_{N})_{w}(\eta)||}\Big\vert.
\end{equation}
By uniformity in the convergence of inequality (\ref{eq2'-00001}), given $\epsilon>0$ there is $t_0$ such that
$\Big\vert a - \frac{1}{t}\log ||(d\phi^t_{N})_{w_n}(\eta_n)\Big\vert<\frac{\epsilon}{2}$ for each $t\geq t_0$  and  all $n$. Moreover, by continuity of $d\phi^t_{N}$, for each $t\geq t_0$ there is $n(t)$ such that $\Big\vert \frac{1}{t}\log \frac{||(d\phi^t_{N})_{w_n}(\eta_n)||}{||(d\phi^t_{N})_{w}(\eta)||}\Big\vert<\frac{\epsilon}{2}$, for each $n\geq n(t)$. Thus 
$$\Big\vert a - \frac{1}{t}\log ||(d\phi^t_{N})_{w}(\eta)|| \Big\vert<\epsilon \ \ \ \text{for each} \ \ t\geq t_0.$$
\end{proof}
\noindent This claim implies, similarly to Corollary \ref{C. R. H. D.},
\begin{equation*}
\lim_{t\rightarrow +\infty}\dfrac{1}{t}\log\left|   \left.\det d_{w}\phi^{t}_{N}\right|_{E^{u}_{w}}\right| =a(n-1) .
\end{equation*}
Let $w\in SN$ be a periodic point of $\phi_{N}^{t}$ of period $\tau$ and $r_{1},\ldots, r_{n-1}$ the set of eigenvalues of $d_{w}\phi^{\tau}_{N}:E^{u}_{w} \rightarrow E^{u}_{w}$ counted with multiplicity. We denote by
\begin{center}
$\displaystyle \rho(w)=\max_{i}\left\lbrace \vert r_{i}\vert\right\rbrace $
\end{center}
the spectral radius. By Gelfand's formula (see \cite{Gelfand}) we have 
 \begin{equation*}
 \lim_{k\rightarrow + \infty}\left. \Vert d_{w}\phi^{\tau k}_{_{N}}\right|_{E^{u}_{w}}\Vert^{\frac{1}{k}}=\rho(w) .
 \end{equation*}
 Thus,
 \begin{equation}\label{tau1}
 \tau a=\lim_{k\rightarrow + \infty}\dfrac{\tau}{\tau k}\log \left. \Vert d_{w}\phi^{\tau k}_{N}\right|_{E^{u}_{w}}\Vert= \log \rho(w).
 \end{equation}
On the other hand,
\begin{equation*}
\dfrac{1}{k}\log \left| \left. \det d_{w}\phi^{\tau k}_{N}\right|_{E^{u}_{w}}\right|  = \sum_{i=1}^{n-1}\log \vert r_{i} \vert .
\end{equation*}
Therefore,
\begin{equation}\label{tau2}
\tau a(n-1)=\lim_{k\rightarrow + \infty}\dfrac{\tau}{\tau k}\log \left| \left. \det d_{w}\phi^{\tau k}_{N}\right|_{E^{u}_{w}} \right|  = \sum_{i=1}^{n-1}\log \vert r_{i} \vert
\end{equation}
From \eqref{tau1} and \eqref{tau2} we conclude that $\log\vert r_{i} \vert=\tau a$ for $1 \le i \le n-1$. So, $\vert r_{i}\vert = \vert r_{j} \vert $ for  $1\le i,j \le n-1$. From Theorem 1.1 at  \cite{butler:butler} we have that  $K_N\equiv -a^2$, which completes the proof of case 1.

\ \\
\textbf{Case 2:} In this case $\sup K_{N}=0$. So, our main goal is to prove that $K_M=K_N=\sup K_N=0$. Since $\sup K_{N}=0$, then  $0\leq \sup K_{N}\leq \inf K_{M}$, in other words, $0\leq K_M$. By hypothesis $M$ has no conjugate points, then $K_{M}\equiv 0$ (cf. \cite{gre:gre} and \cite{hopf:hopf}). So, our goal has been reduced to prove that  $K_{N}\equiv 0$. By contradiction, assume that $N$ is not flat and assume that there is $c>0$ such that $-c^2\leq K_{N}\leq 0$. We have the following claim:\\

\noindent\textbf{Claim:} There is $w\in SN$ such that $$\displaystyle\limsup_{t\to +\infty}\frac{1}{t}\int_{0}^{t}{\rm{Ric}}(\phi^{s}(w)))ds=-B<0,$$
for some $B>0$.
\begin{proof}[\textbf{Proof of Claim}] From Birkhoff Ergodic Theorem, for almost every point $w\in SM$ we have that 

$$\lim_{t\to +\infty}\frac{1}{t}\int_{0}^{t}{\rm{Ric}}(\phi^{s}(w)))ds:=\psi(w),$$
where $\psi(w)$ is a integrable function and 
\begin{equation}\label{EQN1}
\int_{SM}\text{Ric}(\theta)d\mu= \int_{SM}\psi(\theta)d\mu.
\end{equation}
Now, assume by contradiction that for all $w\in SM$, 
$$\displaystyle\limsup_{t\to +\infty}\frac{1}{t}\int_{0}^{t}{\rm{Ric}}(\phi^{s}(w))ds=0.$$
Thus, $\psi(w)=0$ for almost every $\omega\in SM$ and from (\ref{EQN1}) 

$$\int_{SM}\text{Ric}(\theta)d\mu=0.$$
Finally, as $K_N\leq 0$, then $\text{Ric}(\theta)= 0$ for all $\theta\in SM$ and by definition of $\text{Ric}$ we conclude that $K_N=0$ and $N$ is flat, so we have a contradiction and the claim is proven.
\end{proof}

Using the last claim, since $N$ is a compact manifold with non-positive curvature, in particular, $N$ has no focal points, then with similar arguments of Proposition 1 and Corollary 7 at \cite{MelRom:pre2} we have that the unstable Green  subbundle  $G^{u}$ of $TSN$ (see Section \ref{NCP} and \cite{eb:eb} for more details) satisfies 
\begin{equation*}
\limsup_{t\to +\infty} \frac{1}{t}\log|\text{det} \, \, d\phi_{_{N}}^{t}|_{G^u_w}|\geq \frac{B}{c}.
\end{equation*}
As $|\text{det} \, \, d\phi_{_{N}}^{t}|_{G^u_w}|\leq ||d\phi_{_{N}}^{t}|_{G^u_w}||^{\text{dim}{G^u_{w}}}$ and  $\text{dim}\, {G^u_{w}}=\text{dim}\,{N}-1=n-1$, then the last inequality provides that 
\begin{equation}\label{final eq}
\limsup_{t\to +\infty} \frac{1}{t}\log ||d\phi_{N}^{t}|_{E^u_w}||\geq \frac{B}{(n-1)c}.
\end{equation}
To finish our arguments, observe that as $K_M\equiv 0$, then   
$$\displaystyle\lim_{t\to +\infty}\frac{1}{t}\log ||(d\phi^t_{M})_{\theta}(\xi)||=0$$
uniformly in bounded regions of pair $(\theta, \xi)$ for each  $\theta \in SM$ and $\xi\in T_{\theta}SM$.
Proceeding in the same way as in Case 1, we have 
\begin{equation*}\label{eq22'-00001}
\lim_{t\to +\infty}\frac{1}{t}\log ||(d\phi^t_{N})_{w}(\eta)||=0,
\end{equation*}
for each pair $(w, \eta)$, with $w\in SN$ and $\eta \in T_{w}SN$. Thus, the last equation is a contradiction with the inequality (\ref{final eq}), which ends the proof of Case 2 and, consequently, the proof of Theorem.
\end{proof}

\appendix
\section{\\ Comparing distances between a manifold and its submanifolds}
This appendix is devoted to prove the Lemma \ref{comparation of distances}. This a general lemma that can be in other context. 

\begin{lemma}\label{comparation of distances} 
	Let $Q$ be a complete Riemannian manifold and $P$ a compact submanifold of $Q$. Consider $d$ the natural distance in $Q$ and $d_P$ the extrinsic distance in $P$. Then there is $\delta>0$ such that for each $x\in P$
	\begin{equation*}
	d_{P}(x, y)\leq \frac{3}{2}d(x,y) \, \, \, \text{for all} \, \, \, y\in B^{P}_{\delta}(x),
	\end{equation*} 
	where  $B^{P}_{\delta}(x)$ is the ball of center $x$ and radius $\epsilon$ in $P$.
\end{lemma}
\begin{proof}
Since $P$ is a compact submanifold of $Q$, then the injectivity radius $r_{P}$ of $P$ with extrinsic metric is a positive number. Denote by $SP$ the unitary bundle of $P$ and consider the non-negative real function $\mathcal{H}\colon SP\times [0,r_p] \to \mathbb{R}$ defined by 
\[
\mathcal{H}((x,v),t) = 
  \begin{cases}
      \dfrac{d_{P}(\text{exp}^P_{x}\,tv, x)}{d(\text{exp}^P_{x}\, tv, x)}, & t\neq 0 \\
      \, \\
      
     \,\,\,\,\,\,\,\,\,\,\,\,\,\,\,\ 1, & t=0,
  \end{cases}
\]
where $\text{exp}^P_{x}$ denotes the exponential map of $P$.\\
It is easy to see that $\mathcal{H}$ is continuous for all $((x,v),t)$ with $t\neq 0$. We state that $\mathcal{H}$ is continuous at any point $((x,v),0$). To prove that, we used the compactness of $P$ and the following claim:

\textbf{Claim:} For each $x\in P$ the function $\mathcal{H}_{x}\colon T^{1}_{x}P\times [0,r_P] \to \mathbb{R}$ defined by $\mathcal{H}_{x}(v,t):=\mathcal{H}((x,v),t)$ is uniformly continuous function in $T^{1}_{x}P\times [0,r_P]$, where $T^{1}_{x}P$ is the set of unitary tangent vectors of $P$ at $x$.
\begin{proof}[\textbf{Proof of Claim}]
By definition of $\mathcal{H}_{x}$ and compactness of $T^{1}_{x}P\times [0,r_P]$, we only need to prove that  $\ds\lim_{n\to +\infty }\mathcal{H}_{x}(v_n,t_n)=\mathcal{H}_{x}(v,0)=1$, for any sequence $(v_n, t_n)\in T^{1}_{x}P\times [0,r_P]$, which converges to $(v,0)$. In fact: first note that if $(x,v)\in SP$ and  $\alpha(t)=\text{exp}^{P}_{x}\,tv$, then from Lemma \ref{derivative distance} we have that 
		$$\ds \lim_{t\to 0^+}\mathcal{H}((x,v),t))=\ds\lim_{t\to 0^+}\dfrac{\dfrac{d_P(\text{exp}^{P}_{x}\,tv, x)}{t}}{\dfrac{d(\text{exp}^{P}_{x}\,tv, x)}{t}}=\dfrac{\ds\lim_{t\to 0^+}\dfrac{d_{P}(\alpha(t),\alpha(0))}{t}}{\ds\lim_{t\to 0^+}\dfrac{d(\alpha(t),\alpha(0))}{t}}=\frac{||\alpha'(0)||_{P}}{||\alpha'(0)||}=1,$$
since $||\alpha'(0)||_{P}$ is the norm with restrict metric of $P\subset Q$ which is equal to $||\alpha'(0)||$. \\
As $\mathcal{H}_{x}(v_n,0)=1$ then we can assume, without loss of generality, that $t_n\neq 0$ for all $n$.\\
For $n$ large enough, consider the family of $C^1$-curves $\gamma_{n}(t) = \text{exp}^{P}_{x} (t_nv + t t_n(v_n - v))$, $t\in[0,1]$, and note that
			$$ 	d(\text{exp}^{P}_{x}\,t_nv, \text{exp}^{P}_{x}\,t_nv_n)  \leq  d_{P}(\text{exp}^{P}_{x}\,t_nv, \text{exp}^{P}_{x}\,t_nv_n)   \leq \displaystyle\int_{0}^{1} ||\gamma'_{n}(s)||_{P} \, ds  \leq Kt_n||v_n-v||,$$
where $K:=\displaystyle\sup_{\{v\in T_{x}P: ||v||\leq 2r_P\}} || D (\text{exp}^{P}_{x})_{v}||$.
			Hence, since $\ds \lim_{n \to +\infty}\mathcal{H}_{x}(v,t_n)=1$, then 
			 $$d(\text{exp}^{P}_{x}\,t_nv_n, x) \geq d(\text{exp}^{P}_{x}\,t_nv, x) - Kt_n||v_n-v||  = t_n\left(\mathcal{H}_{x}(v,t_n))^{-1}-K||v_n-v||\right) >0.$$  In particular, we have
			$$ 1 \leq \mathcal{H}_{x}(v_n,t_n)) = \dfrac{t_n}{d(\text{exp}^{P}_{x}\, t_nv_n, x)} \leq \dfrac{1}{(\mathcal{H}_{x}(v,t_n)))^{-1} - K||v_n-v||}.$$		
 Therefore
		$\ds \lim_{ n \to + \infty}\mathcal{H}_{x}(v_n,t_n) = 1$ as desired.
\end{proof}
The last claim and compactness of $P$ allow us to conclude the continuity of $\mathcal{H}$ in every point $((x,v),0)$ and consequently the uniformly continuity in $SP\times[0,r_P]$.

To conclude the proof of the lemma, from the uniformly continuity of $\mathcal{H}$,  given $\epsilon=\frac{1}{2}$ there is $\delta$ such that 

$$|t|=d(((x,v),t), ((x,v),0))<\delta\,\, \text{implies} \, \, |\mathcal{H}((x,v),t)-1|< \frac{1}{2}.$$
The last inequality implies that  
\begin{equation}\label{EE2}
|t|< \delta \, \, \, \text{implies}\, \, \, d^s(\text{exp}^s_{x}\,tv, x) \leq \frac{3}{2}d(\text{exp}_{x}\, tv, x).
\end{equation}
Consequently,  if  $\tilde{x}\in B^P_{\delta}(x)$, then there is $\tilde{t}$ with 
$|\tilde{t}|<\delta$, ${v}\in T^1_{x}P$ such that  $\text{exp}^P_{x}\, \tilde{t}{v}=\tilde{x}$. Thus,
(\ref{EE2}) provides the result of the lemma.
\end{proof}
\begin{corollary}\label{COMP DISTANCE}
In the same condition of the \emph{Lemma \ref{comparation of distances}}, there is a constant $\Gamma>1$ such that 
\begin{equation*}
	d_{P}(x, y)\leq \Gamma\cdot d(x,y) \, \, \, \text{for all} \, \, \, x,y\in P.
\end{equation*} 
\end{corollary}
\begin{proof}
For each $x\in P$ consider the function $\Gamma(x)=\displaystyle\sup_{y\neq x}\dfrac{d_{P}(x, y)}{d(x, y)}$. We state that there is $\Gamma>1$ such that $\Gamma(x)\leq \Gamma$ for all $x\in P$. In fact, by contradiction assume that for each $n\in \mathbb{N}$  there is $x_n$ such that $\Gamma_{n}\geq n $. By definition of $\Gamma(x_n)$ there is $y_n\neq x_n$ with $\dfrac{d_{P}(x_n, y_n)}{d(x_n, y_n)}\geq n-1$. Since $P$ is a compact submanifold, then we can assume that $x_n$ converges to $x$ and $y_n$ converges to $y$. From the last inequality, we have that $x=y$. Therefore, from Lemma \ref{comparation of distances}, for n large enough $y_n\in B_{\delta}(x_n)$ and 
$$n-1\leq \dfrac{d_{P}(x_n, y_n)}{d(x_n, y_n)}\leq \frac{3}{2},$$
which provides a contradiction.
\end{proof}

\textbf{Acknowledgments:} The authors would like to thank  Fran\c cois Ledrappier for his useful comments during the preparation of this work and the anonymous referees for the great suggestions that improved the paper. \'Italo Melo was partially supported by FAPEPI (Brazil) and CNPq
(Brazil) and Sergio Romaña was supported by CNPq and Faperj - Bolsa Jovem Cientista do Nosso Estado No. E-26/201.432/2022.
\bibliography{rigiditytheorem}
\bibliographystyle{ijmart}

\end{document}